\documentclass[a4paper,12pt,reqno]{amsart}
\usepackage{latexsym}
\usepackage{amssymb} 
\usepackage{mathrsfs}
\usepackage{amsmath}
\usepackage{latexsym}
\usepackage{delarray}
\usepackage{amssymb,amsmath,amsfonts,amsthm,mathrsfs}
\usepackage[utf8x]{inputenc}

\usepackage{lineno,hyperref}

\newtheorem{thm}{Theorem}[section]
\newtheorem{cor}[thm]{Corollary}
\newtheorem{lem}[thm]{Lemma}

\theoremstyle{definition}

\theoremstyle{remark}
\newtheorem{rem}[thm]{Remark}

\numberwithin{equation}{section}

\newcommand{\ene}{\mathbb{N}}
\newcommand{\ar}{\mathbb{R}}
\newcommand{\cal}{\mathcal}

\newcommand{\bi}{\begin{itemize}}
	\newcommand{\ei}{\end{itemize}}
\newcommand{\be}{\begin{enumerate}}
	\newcommand{\ee}{\end{enumerate}}
\newcommand{\beq}{\begin{equation}}
\newcommand{\eq}{\end{equation}}

\DeclareMathOperator*{\sign}{sign\,}

\begin{document}

%
%
%
%
%
%
%
%
%

\title[On the set of  G\^ateaux differentiability of the $L^1$ norm]{On the set of  G\^ateaux differentiability of the $L^1$ norm}

\author{Julio Delgado}

\address{Universidad del Valle\\
Departamento de Matematicas\\
Calle 13 100-00\\
Cali\\
Colombia}

\email{delgado.julio@correounivalle.edu.co}

\thanks{ The first author was supported by Grant CI71234 Vic. Inv. Universidad del Valle.}
\author{Andr\'es F. Mu\~noz-Tello}
\address{Universidad Santiago de Cali\\
Calle 5 62-00\\
Cali\\
Colombia}
\email{andres.munoz00@usc.edu.co}




\begin{abstract}
 Let $(\Omega, \mathcal{M}, \mu)$ be a measure space. 
In this paper we establish the set of G\^ateaux differentiability for the usual norm  of $L^{1}(\Omega, \mu)$ and the corresponding derivative formulae at each point in this set.  
\end{abstract}

\maketitle

\section{Introduction}
The investigation on the sets of differentiability of Lipschitz functions goes back to the classical Rademacher's Theorem~\cite{ra} of 1919. This theorem states that: 
\begin{center}
	{\it Every Lipschitz function $f:\mathbb{R}^{n}\rightarrow \mathbb{R}$ is differentiable almost everywhere.}
\end{center}
Regarding the converse of Rademacher's Theorem it has been recently established by Preiss and Speight (cf. \cite{ps}) that: if $n>1$ then there exists a Lebesgue null set in $\mathbb{R}^{n}$ containing a point of differentiability of each Lipschitz function $f: \mathbb{R}^{n} \rightarrow \mathbb{R}^{n-1}$. A more detailed discussion around the 
Rademacher's Theorem and some converse  versions can be found in \cite{preim4}. 
An extension of Rademacher's Theorem to functions defined on separable Banach spaces and valued into Banach spaces was established by Phelps in~\cite{ph}: 

\begin{thm}[Phelps]\label{phelps}
	Let $E$ be a separable real Banach space and $F$  a real Banach space with  the 	Radon-Nikod\'{y}m property (RNP). If $T: G \rightarrow F$ is a locally Lipschitz mapping  defined on a nonempty open subset  $G$  of $E$. Then $T$ is G\^ateaux differentiable in the complement of a Gaussian null subset of $G$.	
\end{thm}

Other extensions of Rademacher Theorem to Lipschitz maps between infinite-dimensional Banach spaces have been investigated by several authors e.g. Aronszajn \cite{ar}, Bongiorno~\cite{bon}, Mankiewicz \cite{man}, Phelps~\cite{ph}, Preiss and Zaj\'i\v{c}ek \cite{pz}, among others). The research on the differentiability on Banach spaces has been a very active field  in the last 
decades. We just quote some  few recent examples: In  Potapov and Sukachev's work~\cite{ps14}, they answer to a question in the theory of Schatten- Von Neumann ideals of whether their norms have same differentiability properties as the norms of their commutative counterparts. On the other hand Lindenstrauss et al.~\cite{lin} prove that a real valued Lipschitz function on an Asplund space has points of Fr\'echet differentiability. In in~\cite{gol} Goldys and Gozzi  apply the Phelps Theorem in the proof of existence and uniqueness of solutions of second order parabolic Hamilton-Jacobi-Bellman equations in Hilbert spaces, obtaining optimal feedback for first order stochastic PDEs that arise in economic theory, the theory of population dynamics and financial models.\\

The study of the set of differentiability of Lipschitz functions on Banach spaces has been carried out for instance in Deville et al.~\cite{de}, Diestel~\cite{di}, Dore and Maleva~\cite{dore} among others. For the $L^{p}(\Omega)$ space where $\Omega$ is a bounded domain and $1<p<\infty$, it is well known (cf. \cite{di}) that the norm $\|f\|_{L^{p}} = (\int_{\Omega}|f(w)|^{p}dw)^{1/p}$ is Fr\'echet differentiable and moreover its derivative is given by 
\beq\partial_{F}\|f\|_{L^{p}}=\frac{|f|^{p-1}\sign(f)}{\|f\|_{L^{p}}^{p-1}},\label{formF1q}\eq
for any $f\neq 0$. The case of Fr\'echet differentiability for $p=1$  is more pathological. Here we will explicitly find the set of  G\^ateaux differentiability for the  $L^{1}(\Omega, \mu)$ norm, where $(\Omega,\cal{M}, \mu)$ is a measure space. The characterization of this set is given in Theorem \ref{main1a} of Section 3.



\section[Preliminaries]{Preliminaries}
In this section we briefly give a basic background in our context. We start with 
 the  two most classical notions of differentiability on Banach spaces.  The most simple one is that of G\^ateaux differentiability. Let $X$, $Y$ be  Banach spaces and $\Omega\subseteq X$ an open set. A function $f:\Omega\rightarrow Y$ is {\it G\^ateaux differentiable ($G$-differentiable)} at a point $a\in\Omega$, if 
$$\lim\limits_{t\rightarrow 0} \frac{f(a+th)-f(a)}{t}=u \left(h\right)$$
exists for all $h\in X$, and $u$ defines  a bounded linear operator from $X$ into $Y$. The operator $u$ is called the G\^{a}teaux derivative of $f$ at the point $a$ and is denoted by $\partial_{G}f(a)$. 

A function $f:\Omega \rightarrow Y$ is {\it Fr\'echet differentiable ($F$-differentiable)} at $a$, if there is a bounded linear operator $u$ such that
$$f\left(a+h\right)-f\left(a\right)-u\left(h\right)=r\left(h\right)$$ where $\lim\limits_{h\rightarrow 0} \frac{\left\|r\left(h\right)\right\|_{Y}}{\left\| h \right\|_{X}}=0$. The operator $u$ is called the Fr\'{e}chet derivative of $f$ at the point $a$  and is denoted by $\partial_{F}f(a)$.

It is not difficult to see that, if $f$ is   $F$-differentiable at $a$ then $f$ is $G$-differentiable at $a$ and in that case $\partial_{F}f(a)=\partial_{G}f(a)$.  For a given  function $\varphi:X\rightarrow Y$, we say that $B$ is the set of $G$-differentiability or $F$-differentiability of $\varphi$, if respectively for every element of the set $B$ the function $\varphi$ is $G$-differentiable or $F$-differentiable.

On the other hand, a set is called a {\it $G_{\delta}$ set} if it can be expressed as a countable intersection of open sets. This type of set is fundamental in the definition of the {\it F-Asplund  and $G$-Asplund} spaces, since these respectively are Banach spaces in which all convex and continuous function defined in an open and convex subset is $F$-differentiable (resp. $G$-differentiable) in a dense and $G_{\delta}$ set. The following theorem shows  important examples of $G$-Asplund spaces.

\begin{thm}[Mazur]\label{a4}
	Every separable Banach space is a $G$-Asplund space.
\end{thm}

Moreover, it is well known that the set of all $F$-Asplund spaces is strictly contained in the set of all $G$-Asplund spaces. In fact, as proved in Deville et al.~\cite{de}, the usual norm of $\ell^{1}(\mathbb{R})$ is not $F$-differentiable at  any point of the space $\ell^{1}(\mathbb{R})$. Then $\ell^{1}(\mathbb{R})$ is not a $F$-Asplund space; but $\ell^{1}(\mathbb{R})$ is a separable Banach space and therefore a $G$-Asplund space by  Mazur's theorem.

As we have quoted before, the $L^p$ norm is $F$-differentiable at any point in the case  $1<p<\infty$ and its corresponding derivative is given by \eqref{formF1q}. 
 Moreover, since the dual of $L^{p}(\Omega)$ is $L^{q}(\Omega)$, for $1<p<\infty $ and $\frac{1}{p}+\frac{1}{q}=1$, and they are separable Banach spaces, by  Theorem 2.12 in Phelps~\cite{ph2} the space $L^{p}(\Omega)$ is a $F$-Asplund space for $1<p<\infty$.


\begin{rem} In the special case of the $L^{1}(\Omega,\cal{M}, \mu)$ norm we observe the following regarding  the set of $G$-differentiability. It is well known that(cf. \cite{cohn})  $L^1(\Omega,\cal{M}, \mu)$ is separable provided $\mu$ is $\sigma$-finite and the  $\sigma$-algebra $\cal{M}$ is countably generated. Since $\mathbb{R}$ has the RNP, then the $L^{1}(\Omega, \mu)$ norm satisfies the conditions of  Phelps's Theorem. Therefore, the Gaussian measure of the complement of the $G$-differentiability set is zero and  by Mazur's Theorem this $G$-differentiability set is dense and a $G_{\delta}$ set.	
\end{rem}	

In the next section we will give an explicit description of the set of $G$-differentiability for the  usual norm of $L^{1}(\Omega, \cal{M}, \mu)$ without assuming  the condition of separability.

\section[The $G$-differentiability set for the $L^{1}(\Omega, \mu)$ norm]{The $G$-differentiability set for the $L^{1}(\Omega, \mu)$ norm}

Given a measure space $(\Omega, \cal{M}, \mu)$, we denote by $\mathcal{L}^1(\Omega, \mu)$ the space of the real valued integrable functions over $\Omega$.  On $\mathcal{L}(\Omega, \mu)$ one defines the pseudonorm $\varphi(f)=\int_{\Omega}|f(x)|d\mu(x)$. One also defines on $\mathcal{L}(\mu)$  an equivalence relation $\mathcal{R}$  by, $f\mathcal{R}g\Longleftrightarrow f = g\, \mu-a.e.$ The equivalence class of $f$ is denoted by $\tilde{f}.$ Thus, the quotient $\mathcal{L}(\Omega)/ \mathcal{R}$ defines the classical Banach space $L^{1}(\Omega,\mu)$ endowed with the norm $\varphi(\tilde{f})=\int_{\Omega}|f(x)|d\mu(x)$.\\

We also recall recall that the signature $\sign(x)$ of a real number $x$, is defined by  $\sign(x)=\frac{x}{|x|}$ if $x\neq 0$, and $\sign(x)=0$ if $x=0$. 

In order to prepare the proof of the main theorem we first establish a mild lemma which will help to obtain a pointwise formulae for the  G\^ateaux derivative. 

\begin{lem}\label{lmain1a} Let $f,\, h:\Omega\rightarrow \ar$  be functions and $x\in\Omega$ such that $h(x)\neq 0$. Then 
\beq\lim\limits_{t\rightarrow 0}\frac{|f(x)+th(x)|-|f(x)|}{t},\label{j6nr}\eq
exists if and only if $f(x)\neq 0$.\\
Moreover, if that is the case
\beq\lim\limits_{t\rightarrow 0}\frac{|f(x)+th(x)|-|f(x)|}{t}=\sign(f(x))h(x).\label{ftja46}\eq
\end{lem}
\begin{proof} If the limit \eqref{j6nr} exists  and  $f(x)=0$, then the limit  
\[\lim\limits_{t\rightarrow 0}\frac{|th(x)|}{t}=\lim\limits_{t\rightarrow 0}\frac{|t||h(x)|}{t}\]
does not exists, which is a contradiction, hence $f(x)\neq 0$. \\

We now assume $f(x)\neq 0$. Since $h(x)\neq 0$, we can choose $\delta=\frac{|f(x)|}{2|h(x)|}$ and observe that for  $|t|<\delta$ we have 
 \beq \sign(f(x)+th(x))=\sign(f(x)).\label{sig8dd2}\eq
Hence
\begin{eqnarray*}
\frac{|f(x)+th(x)|-|f(x)|}{t}&=&\frac{\sign(f(x)+th(x))(f(x)+th(x))-\sign(f(x))f(x)}{t}\\
&=&\frac{\sign(f(x))t h(x)}{t}\\
&=&\sign(f(x)) h(x).
\end{eqnarray*}
Therefore
\[\lim\limits_{t\rightarrow 0}\frac{|f(x)+th(x)|-|f(x)|}{t}=\sign(f(x))h(x).\]
\end{proof}

We now consider a measure space $(\Omega, \cal{M},\mu)$. As a consequence of the  lemma above we obtain  the following corollary.
\begin{cor} \label{cordd2} Let $f, h:\Omega\rightarrow \ar$ be measurable functions such that $h(x)\neq 0$ $\mu$-a.e. Then 
\beq\lim\limits_{t\rightarrow 0}\frac{|f(x)+th(x)|-|f(x)|}{t} \mbox{ exists } \mu-a.e\label{limj57}\eq
if and only if $f(x)\neq 0$ $\mu$-a.e. Moreover, if that is the case the formula \eqref{ftja46} holds.
\end{cor}
\begin{proof} We suppose that $h(x)\neq 0$ for all $x\in A$, with $\mu(A^c)=0$ and first  assume that  the limit \eqref{limj57} exists $\mu$-a.e. Hence, $f(x)\neq 0$ for every  $x\in A$ by Lemma \ref{lmain1a},  thus  $f(x)\neq 0$  $\mu$-a.e.  \\
On the other hand, if $f(x)\neq 0$  $\mu$-a.e. Then there exists a  set $D\in \cal{M}$ such that $\mu(D^c)=0$ and  $f(x)\neq 0$ for all $x\in D$. By taking $x\in A\cap D$ we have that $f(x), h(x)\neq 0$. Then by Lemma \ref{lmain1a} the limit \eqref{limj57} exists for every $x\in A\cap D$. Therefore the limit \eqref{limj57} 
 exists $\mu$-a.e.\\
 The last part of the statement clearly follows from the last part of Lemma \ref{lmain1a} by applying the formula \eqref{ftja46} in a set $\Gamma$ where we have $f(x), h(x)\neq 0$ and $\mu(\Gamma^c)=0.$
\end{proof}
We are now ready to present our main result. We consider a measure space $(\Omega, \cal{M},\mu)$ and the  following condition:\\

(A1)\,\,For every $ I\in\cal{M}$   with $ \mu(I)=\infty,$ there exists a set $F\subset I$  with $F\in\cal{M}$    such that 
\[ 0<\mu(F)<\infty.\]

Every  $\sigma$-finite measure that is not the trivial null measure satisfies the condition (A1). The counting measure on the power set of $\ar$ satisfies the condition (A1) and is not $\sigma$-finite.
 On the other hand, every non-atomic measure also satisfies the condition (A1).\\ 
 
The following theorem 
gives a precise description of the $G$-differentiability set for the usual  $L^{1}(\Omega,\mu)$ norm and the corresponding formulae for the  G\^ateaux derivative.\\

Let $G$ be the subset of $L^{1}(\Omega,\mu)$ consisting of equivalence classes $\tilde{f}$ such that $f\neq 0$ $\mu$-a.e., for some representative $f$ of the  class $\tilde{f}$. It is clear  that if $\tilde{f}\in G$, any representative of the class $\tilde{f}$ satisfies the condition defining $G$. 
\begin{thm}\label{main1a} Let $(\Omega, \cal{M},\mu)$ be a measure space satisfying the condition (A1). Let $B$ be the subset of $L^{1}(\Omega,\mu)$ above defined. Then the set of $G$-differentiability  of the norm $\varphi$ is $G$. Moreover, the G\^ateaux derivative of $\varphi$ for $\tilde{f}\in G$ and the direction $h\in L^{1}(\Omega,\mu)$ is given by 
	\beq\partial_{G}\varphi(\tilde{f})(h)=\int_{\Omega}\sign(f(x))h(x)d\mu(x),\label{Gder67}\eq
where $f$ is any representative of the class  $\tilde{f}$.	
	
\end{thm}

\begin{proof} 
We assume $\tilde{f}\in G$ and that $f$ is a representative of the class $\tilde{f}$. We start by observing that for any $h\in  L^{1}(\Omega,\mu)$ the right hand side of the  formula \eqref{Gder67} is independent of the representative in the class $\tilde{f}$. \\
	
We are now going to prove the $G$-differentiability of $\varphi$ at $\tilde{f}$. Let $h\in  L^{1}(\Omega,\mu)\setminus\{0\}$, an application of the  Lebesgue dominated convergence Theorem gives us
\begin{eqnarray*}
	\lim\limits_{t\rightarrow 0}\frac{\varphi(f+th)-\varphi(f)}{t}&=&\lim\limits_{t\rightarrow 0} \frac{\int\limits_{\Omega}|f(x)+th(x)|d\mu(x)-\int\limits_{\Omega}|f(x)|d\mu(x)}{t}\\
&=&\lim\limits_{t\rightarrow 0}\, \frac{\int\limits_{\Omega}\left(|f(x)+th(x)|-|f(x)|\right)d\mu(x)}{t}\\
&=&\int\limits_{\Omega}\lim\limits_{t\rightarrow 0} \frac{\,\left(|f(x)+th(x)|-|f(x)|\right)}{t}d\mu(x).\\
	 	\end{eqnarray*} 
Since $h(x)\neq 0$ $\mu$-a.e. and  $f(x)\neq 0$ $\mu$-a.e., by applying Corollary \ref{cordd2} we have

\[\lim\limits_{t\rightarrow 0}\frac{|f(x)+th(x)|-|f(x)|}{t}=\sign(f(x))h(x)\]
$\mu$-a.e., and

 \[\lim\limits_{t\rightarrow 0}\frac{\varphi(f+th)-\varphi(f)}{t}=\int_{\Omega}\sign(f(x))h(x)d\mu(x).\]
	
We now assume that $\tilde{f}\in G^c$ and we will show that $\varphi$ does not have $G$-derivative in $\tilde{f}$. For the set $Z_1=\{x\in\Omega:f(x)=0\}$ we have
$\mu(Z_1)>0.$ If $\mu(Z_1)=\infty$ by condition (A1) we can choose a subset $Z$ of $Z_1$ such that $0<\mu(Z)<\infty.$ We consider $h=1_{Z}$, thus $h\in L^{1}(\Omega, \mu)\setminus 0$. Since $h=0$ on $Z^c$, for $t\in\ar$ we have  $|f+th|-|f|=|f|-|f|=0$ on  $Z^c$. Hence, for the calculation of the $G$-derivative of $\varphi$ at $f$ in the direction $h$, is enough to consider the integration on $Z$.
 For $t\neq 0$ we have 
		\begin{eqnarray*}
		\frac{\varphi(f+th)-\varphi(f)}{t}&=& \frac{\left(\int_{Z}|f(x)+th(x)|d\mu(x)-\int_{Z}|f(x)|d\mu(x)\right)}{t} \nonumber\\ 
		 &=&\frac{ \int_{Z}|t|d\mu(x)}{t}\nonumber\\
		 &=&\frac{|t|}{t}\mu(Z)\nonumber\\
		 &=&\sign (t)\mu(Z).\nonumber\\
		\end{eqnarray*}
Hence,  we have $(\partial_G^+\varphi)f(h)=\mu(Z)$ and $(\partial_G^+\varphi)f(h)=-\mu(Z)$. Therefore  the $G$-derivative of $\varphi$ at $f$ does not exists.\\
 
It is clear that the formula \eqref{Gder67} defines a bounded linear operator from  $ L^{1}(\Omega, \mu)$ into $\ar$ for every $\tilde{f}$ in $G$, and this concludes the proof of the theorem. 
\end{proof}

\begin{rem} (1) In the case $\Omega=\mathbb{N}$ and the counting measure, Theorem \ref {main1a} absorbs a more classical result on the G\^ateaux differentiability of the  $\ell^1(\ene)$-norm, namely that the set of  G\^ateaux differentiability of the $ \ell^1(\ene)$-norm is the  set $G\subset \ell^1(\ene)$ of sequences $(x_n)_n$ such that $x_n\neq 0$ for all $n$, and the G\^ateaux derivative of $\varphi$ at $x$ and in the direction $h$ is given by
\[\partial_{G}\varphi(x)(h)=\sum\limits_n\sign(x_n)h_n.\]
  
(2) The condition (A1) allows $L^1$ to be be non-separable. 
Therefore, Theorem \ref {main1a} gives a description of the set of differentiability  beyond the scope of Phelps's  and Mazur's theorem. \\
\end{rem}	
In the special separable case, as a consequence of Theorem \ref{main1a} and Phelps Theorem \ref{phelps} we have:
\begin{cor} Let $(\Omega,\cal{M}, \mu)$ be a measure space satisfying the  condition (A1) and such that $L^1(\Omega,\cal{M}, \mu)$ is separable. Then, the set of  G\^ateaux differentiability of the $L^1$-norm is $G$, $G$ is a  $G_{\delta}$ set and $G^c$ is a Gaussian null set.  
\end{cor}

As quoted before $L^1(\Omega,\cal{M}, \mu)$ is separable provided $\mu$ is $\sigma$-finite and the  $\sigma$-algebra $\cal{M}$ is countably generated. If the measure $\mu$ is not the trivial null measure then $\mu$ satisfies the condition (A1). 






\end{document}